\documentclass[12pt,reqno]{amsart} 
\usepackage{amssymb,amscd,url}

\begin{document}

\allowdisplaybreaks

\newif\ifdraft 
\drafttrue
\draftfalse
\newcommand{\DRAFTNUMBER}{1}
\newcommand{\DATE}{\today\ \ifdraft(Draft \DRAFTNUMBER)\fi}
\newcommand{\TITLE}{An algebraic approach to certain\\
              cases of Thurston rigidity}
\newcommand{\TITLERUNNING}{An algebraic approach to Thurston rigidity}

\hyphenation{ca-non-i-cal archi-me-dean non-archi-me-dean}


\newtheorem{theorem}{Theorem}
\newtheorem{lemma}[theorem]{Lemma}
\newtheorem{conjecture}[theorem]{Conjecture}
\newtheorem{proposition}[theorem]{Proposition}
\newtheorem{corollary}[theorem]{Corollary}
\newtheorem*{claim}{Claim}

\theoremstyle{definition}
\newtheorem{question}{Question}
\renewcommand{\thequestion}{\Alph{question}} 
\newtheorem*{definition}{Definition}
\newtheorem{example}[theorem]{Example}
\newtheorem{remark}[theorem]{Remark}

\theoremstyle{remark}
\newtheorem*{acknowledgement}{Acknowledgements}



\newenvironment{notation}[0]{%
  \begin{list}%
    {}%
    {\setlength{\itemindent}{0pt}
     \setlength{\labelwidth}{4\parindent}
     \setlength{\labelsep}{\parindent}
     \setlength{\leftmargin}{5\parindent}
     \setlength{\itemsep}{0pt}
     }%
   }%
  {\end{list}}

\newenvironment{parts}[0]{%
  \begin{list}{}%
    {\setlength{\itemindent}{0pt}
     \setlength{\labelwidth}{1.5\parindent}
     \setlength{\labelsep}{.5\parindent}
     \setlength{\leftmargin}{2\parindent}
     \setlength{\itemsep}{0pt}
     }%
   }%
  {\end{list}}
\newcommand{\Part}[1]{\item[\upshape#1]}

\renewcommand{\a}{\alpha}
\renewcommand{\b}{\beta}
\newcommand{\g}{\gamma}
\renewcommand{\d}{\delta}
\newcommand{\e}{\epsilon}
\newcommand{\f}{\varphi}
\newcommand{\bfphi}{{\boldsymbol{\f}}}
\renewcommand{\l}{\lambda}
\renewcommand{\k}{\kappa}
\newcommand{\lhat}{\hat\lambda}
\newcommand{\m}{\mu}
\newcommand{\bfmu}{{\boldsymbol{\mu}}}
\renewcommand{\o}{\omega}
\renewcommand{\r}{\rho}
\newcommand{\rbar}{{\bar\rho}}
\newcommand{\s}{\sigma}
\newcommand{\sbar}{{\bar\sigma}}
\renewcommand{\t}{\tau}
\newcommand{\z}{\zeta}

\newcommand{\D}{\Delta}
\newcommand{\G}{\Gamma}
\newcommand{\F}{\Phi}
\renewcommand{\L}{\Lambda}

\newcommand{\ga}{{\mathfrak{a}}}
\newcommand{\gA}{{\mathfrak{A}}}
\newcommand{\gb}{{\mathfrak{b}}}
\newcommand{\gm}{{\mathfrak{m}}}
\newcommand{\gn}{{\mathfrak{n}}}
\newcommand{\go}{{\mathfrak{o}}}
\newcommand{\gO}{{\mathfrak{O}}}
\newcommand{\gp}{{\mathfrak{p}}}
\newcommand{\gP}{{\mathfrak{P}}}
\newcommand{\gq}{{\mathfrak{q}}}
\newcommand{\gR}{{\mathfrak{R}}}

\newcommand{\Abar}{{\bar A}}
\newcommand{\Ebar}{{\bar E}}
\newcommand{\Kbar}{{\bar K}}
\newcommand{\Pbar}{{\bar P}}
\newcommand{\Sbar}{{\bar S}}
\newcommand{\Tbar}{{\bar T}}
\newcommand{\ybar}{{\bar y}}
\newcommand{\phibar}{{\bar\f}}

\newcommand{\Acal}{{\mathcal A}}
\newcommand{\Bcal}{{\mathcal B}}
\newcommand{\Ccal}{{\mathcal C}}
\newcommand{\Dcal}{{\mathcal D}}
\newcommand{\Ecal}{{\mathcal E}}
\newcommand{\Fcal}{{\mathcal F}}
\newcommand{\Gcal}{{\mathcal G}}
\newcommand{\Hcal}{{\mathcal H}}
\newcommand{\Ical}{{\mathcal I}}
\newcommand{\Jcal}{{\mathcal J}}
\newcommand{\Kcal}{{\mathcal K}}
\newcommand{\Lcal}{{\mathcal L}}
\newcommand{\Mcal}{{\mathcal M}}
\newcommand{\Ncal}{{\mathcal N}}
\newcommand{\Ocal}{{\mathcal O}}
\newcommand{\Pcal}{{\mathcal P}}
\newcommand{\Qcal}{{\mathcal Q}}
\newcommand{\Rcal}{{\mathcal R}}
\newcommand{\Scal}{{\mathcal S}}
\newcommand{\Tcal}{{\mathcal T}}
\newcommand{\Ucal}{{\mathcal U}}
\newcommand{\Vcal}{{\mathcal V}}
\newcommand{\Wcal}{{\mathcal W}}
\newcommand{\Xcal}{{\mathcal X}}
\newcommand{\Ycal}{{\mathcal Y}}
\newcommand{\Zcal}{{\mathcal Z}}

\renewcommand{\AA}{\mathbb{A}}
\newcommand{\BB}{\mathbb{B}}
\newcommand{\CC}{\mathbb{C}}
\newcommand{\FF}{\mathbb{F}}
\newcommand{\GG}{\mathbb{G}}
\newcommand{\NN}{\mathbb{N}}
\newcommand{\PP}{\mathbb{P}}
\newcommand{\QQ}{\mathbb{Q}}
\newcommand{\RR}{\mathbb{R}}
\newcommand{\ZZ}{\mathbb{Z}}

\newcommand{\bfa}{{\mathbf a}}
\newcommand{\bfb}{{\mathbf b}}
\newcommand{\bfc}{{\mathbf c}}
\newcommand{\bfe}{{\mathbf e}}
\newcommand{\bff}{{\mathbf f}}
\newcommand{\bfg}{{\mathbf g}}
\newcommand{\bfp}{{\mathbf p}}
\newcommand{\bfr}{{\mathbf r}}
\newcommand{\bfs}{{\mathbf s}}
\newcommand{\bft}{{\mathbf t}}
\newcommand{\bfu}{{\mathbf u}}
\newcommand{\bfv}{{\mathbf v}}
\newcommand{\bfw}{{\mathbf w}}
\newcommand{\bfx}{{\mathbf x}}
\newcommand{\bfy}{{\mathbf y}}
\newcommand{\bfz}{{\mathbf z}}
\newcommand{\bfA}{{\mathbf A}}
\newcommand{\bfF}{{\mathbf F}}
\newcommand{\bfB}{{\mathbf B}}
\newcommand{\bfD}{{\mathbf D}}
\newcommand{\bfG}{{\mathbf G}}
\newcommand{\bfI}{{\mathbf I}}
\newcommand{\bfM}{{\mathbf M}}
\newcommand{\bfzero}{{\boldsymbol{0}}}

\newcommand{\Adele}{\textsf{\upshape A}}
\newcommand{\Aut}{\operatorname{Aut}}
\newcommand{\Br}{\operatorname{Br}}  
\newcommand{\Crit}{\operatorname{Crit}} 
\newcommand{\crit}{{\textup{crit}}}
\newcommand{\Disc}{\operatorname{Disc}}
\newcommand{\Div}{\operatorname{Div}}
\newcommand{\End}{\operatorname{End}}
\newcommand{\Fbar}{{\bar{F}}}
\newcommand{\Fix}{\operatorname{Fix}}
\newcommand{\FOD}{\textup{FOM}}
\newcommand{\FOM}{\textup{FOD}}
\newcommand{\FilledJulia}{\mathcal{K}}
\newcommand{\Gal}{\operatorname{Gal}}
\newcommand{\ghat}{{\hat g}}
\newcommand{\GL}{\operatorname{GL}}
\newcommand{\Index}{\operatorname{Index}}
\newcommand{\Image}{\operatorname{Image}}
\newcommand{\hhat}{{\hat h}}
\newcommand{\Julia}{\mathcal{J}}
\newcommand{\Lattes}{{\operatorname{Lat}}}
\newcommand{\liftable}{{\textup{liftable}}}
\newcommand{\Ksep}{K^{\textup{sep}}}
\newcommand{\Ker}{{\operatorname{ker}}}
\newcommand{\Lsep}{L^{\textup{sep}}}
\newcommand{\Lift}{\operatorname{Lift}}
\newcommand{\LS}[2]{{\genfrac{(}{)}{}{}{#1}{#2}}} 
\newcommand{\vlim}{\operatornamewithlimits{\text{$v$}-lim}}
\newcommand{\wlim}{\operatornamewithlimits{\text{$w$}-lim}}
\newcommand{\MOD}[1]{~(\textup{mod}~#1)}
\newcommand{\Moduli}{{\operatorname{\textup{\textsf{M}}}}}
\newcommand{\Norm}{{\operatorname{\mathsf{N}}}}
\newcommand{\notdivide}{\nmid}
\newcommand{\normalsubgroup}{\triangleleft}
\newcommand{\odd}{{\operatorname{odd}}}
\newcommand{\onto}{\twoheadrightarrow}
\newcommand{\Orbit}{\mathcal{O}}
\newcommand{\ord}{\operatorname{ord}}
\newcommand{\Per}{\operatorname{Per}}
\newcommand{\PrePer}{\operatorname{PrePer}}
\newcommand{\PGL}{\operatorname{PGL}}
\newcommand{\Pic}{\operatorname{Pic}}
\newcommand{\Prob}{\operatorname{Prob}}
\newcommand{\Qbar}{{\bar{\QQ}}}
\newcommand{\rank}{\operatorname{rank}}
\newcommand{\Rat}{\operatorname{Rat}}
\newcommand{\Resultant}{\operatorname{Res}}
\renewcommand{\setminus}{\smallsetminus}
\newcommand{\Span}{\operatorname{Span}}
\newcommand{\tors}{{\textup{tors}}}
\newcommand{\Trace}{\operatorname{Trace}}
\newcommand{\twistedtimes}{\mathbin{%
   \mbox{$\vrule height 6pt depth0pt width.5pt\hspace{-2.2pt}\times$}}}
\newcommand{\UHP}{{\mathfrak{h}}}    
\newcommand{\Wreath}{\operatorname{Wreath}}
\newcommand{\wt}{\operatorname{wt}} 
\newcommand{\<}{\langle}
\renewcommand{\>}{\rangle}

\newcommand{\longhookrightarrow}{\lhook\joinrel\longrightarrow}
\newcommand{\longonto}{\relbar\joinrel\twoheadrightarrow}


\title[\TITLERUNNING]{\TITLE}
\date{\DATE}

\author{Joseph H. Silverman}
\email{jhs@math.brown.edu}
\address{Mathematics Department, Box 1917
         Brown University, Providence, RI 02912 USA}

\thanks{
     The author's research is supported by 
     NSF DMS-0650017 and DMS-0854755.
}

\begin{abstract}
In the moduli space of polynomials of degree~$3$ with marked critical
points~$c_1$ and~$c_2$, let~$C_{1,n}$ be the locus of maps for
which~$c_1$ has period~$n$ and let~$C_{2,m}$ be the locus of maps for
which~$c_2$ has period~$m$. A consequence of Thurston's rigidity
theorem is that the curves~$C_{1,n}$ and~$C_{2,m}$ intersect
transversally.  We give a purely algebraic proof that the intersection
points are~$3$-adically integral and use this to prove transversality.
We also prove an analogous result when~$c_1$ or~$c_2$ or both are
taken to be preperiodic with tail length exactly~$1$.
\end{abstract}


\maketitle

\section{Introduction}

The moduli space~$\Pcal_d$ of polynomials of degree~$d\ge2$ is the space
of polynomials modulo conjugation by the affine action $z\mapsto\a
z+\b$.  Working over~$\CC$ and choosing appropriate values for~$\a$
and~$\b$, every polynomial can be put into the form
\[
  f(z) = z^d + a_2z^{d-2}+\cdots+a_d,
\]
so~$\Pcal_d\cong\CC^{d-1}$. The polynomial~$f$ has~$d-1$ critical
points (counted with multiplicity), and we write~$\Pcal_d^\crit$ for
the moduli space of polynomials~$f$ with marked critical
point~$(c_1,\ldots,c_{d-1})$. Imposing natural relations on these
critical points gives subvarieties of~$\Pcal_d^\crit$, and an important
consequence of Thurston's rigidity theorem~\cite{MR1251582} is that in many
cases these subvarieties have transversal intersection. For example,
transversality holds if we require~$c_1,\ldots,c_{d-1}$ to be periodic
with respective periods~$n_1,\ldots,n_{d-1}$, or more generally if
they are preperiodic with specified tail lengths and
periods. Thurston's theorem also gives analogous results for rational
functions.

The proof of Thurston's general theorem is quite difficult and
requires deep tools; see~\cite{MR1251582}. Adam Epstein has asked if
one might prove at least some cases of Thurston rigidity using
$p$-adic and/or algebraic methods. In this note we give an algebraic
$3$-adic proof of the following special case of Thurston
rigidity for cubic polynomials.

\begin{theorem}
\label{thm:cubicrigid}
Let $\Pcal_3^\crit$ be the moduli space of polynomials of degree~$3$
with marked critical points, i.e., points in $\Pcal_3^\crit$ are
equivalence classes of triples $(f,c_1,c_2)$, where~$f\in\CC[z]$ is a
polynomial of degree~$3$ and~$c_1$ and~$c_2$ are the critical points
of~$f$.
\par
For integers $n,m\ge1$, let
\begin{align*}
  C_{1,n} &= \bigl\{(f,c_1,c_2)\in\Pcal_3^\crit :
       \text{$c_1$ is periodic with $f^n(c_1)=c_1$} \bigr\}, \\
  C_{2,m} &= \bigl\{(f,c_1,c_2)\in\Pcal_3^\crit :
       \text{$c_2$ is periodic with $f^m(c_2)=c_2$} \bigr\}.
\end{align*}
Then $C_{1.n}$ and $C_{2,m}$ intersect transversally at all of their
points of intersection.
\end{theorem}

Our proof of Theorem~\ref{thm:cubicrigid} may be compared with the
analogous $2$-adic proof for quadratic polynomials that is due
independently to Adler and Gleason; see~\cite[Lemma~19.1]{DouHubNotes}
and~\cite{MR1620850}, and also \cite[Appendix]{arxiv1010.2780} for a
generalization.  Our proof of Theorem~\ref{thm:cubicrigid} may also be
compared with the recent, independently discovered, $3$-adic proof by
Epstein~\cite{arxiv1010.2780}.  (We note that Epstein's paper contains
results stronger than our Theorem~\ref{thm:cubicrigid}. The primary
purpose of our paper is to provide a conceptually different proof.)
Both our proof and Epstein's proof deduce the final conclusion, namely
that a certain Jacobian determinant is non-zero, by showing that it
does not vanish modulo~$3$.  The most difficult part of the proof is
to show that the critical points of suitably normalized
post-critically finite cubic polynomials are $3$-adically integral,
and this is where the two proofs differ. Epstein's proof uses a
detailed analysis of the sequence of $3$-adic valuations
$\ord_3\bigl(f^n(c)\bigr)$ of the points in the forward orbit of a
critical point. Thus it makes extensive use of a ($p$-adic) metric
and has a dynamical flavor. Our proof uses an
estimate for the degrees of the curves~$C_{1,n}$ and~$C_{2,m}$,
followed by a resultant calculation, so is much more algebraic in
nature. We mention in particular the interesting explicit formula
(Lemma~\ref{lemma:resxpnxpm}) for the resultant
\[
  \Resultant(x^{p^n}-x-A,x^{p^m}-x-B) \in \FF_p[A,B].
\]
This formula is used to show (Theorem~\ref{prop:ResFnGm3int}) that a
certain resultant has maximal degree by showing that it has maximal
degree when reduced modulo~$3$.

Thurston's theorem deals also with the case that the critical points
are preperiodic, i.e., have finite orbits. Using an algebraic trick,
we are able to give an algebraic proof of this result for cubic
polynomials in the case that the critical points have tail length at
most~$1$.  We give the exact statement and proof in
Section~\ref{section:preperiodic}. It would be quite interesting to
extend this result to allow arbitrary preperiodic behavior.

\section{Proof of Thurston Rigidity for Cubic Polynomials}
In this section we give the proof of Theorem~\ref{thm:cubicrigid}.
Making a change of variables, we may assume that our cubic
polynomials have the form
\[
  f_{x,y}(z)= z^3 - 3x^2z + y
\]
with marked critical points $\pm x$.  For the given integers
$n,m\ge1$, we let
\begin{equation}
  \label{eqn:defFG}
  F^{(n)}(x,y) = f_{x,y}^n(x)-x
  \quad\text{and}\quad
  G^{(m)}(x,y) = f_{x,y}^m(-x)+x.
\end{equation}
Then the solutions to
\begin{equation}
  \label{eqn:F=G=0}
  F^{(n)}(x,y)=G^{(m)}(x,y)=0
\end{equation}
are exactly the pairs $(\a,\b)$ with the property that the critical
points of~$f_{\a,\b}(z)$ have period~$n$ and~$m$, respectively.
\par
Let $(\a,\b)\in\CC$ be a solution to~\eqref{eqn:F=G=0}. The curves~$F^{(n)}=0$
and~$G^{(m)}=0$ are \emph{transversal} at~$(\a,\b)$ if 
and only if the Jacobian determinant does not vanish, i.e., 
\[
  \det\begin{pmatrix} 
    F^{(n)}_x(\a,\b) & G^{(m)}_x(\a,\b) \\
    F^{(n)}_y(\a,\b) & G^{(m)}_y(\a,\b) \\
   \end{pmatrix} \ne 0.
\]
\par
In general, the Jacobian determinant is the polynomial
\begin{equation}
  \label{eqn:Jxy}
  J(x,y) = 
  \det\begin{pmatrix} 
    F^{(n)}_x(x,y) & G^{(m)}_x(x,y) \\
    F^{(n)}_y(x,y) & G^{(m)}_y(x,y) \\
   \end{pmatrix}\in\ZZ[x,y].
\end{equation}
Then the curves~$F^{(n)}=0$ and~$G^{(m)}=0$ intersect transversally at
all of their intersection points if and only if the ideal
\[
  \bigl(F^{(n)}(x,y),G^{(m)}(x,y),J(x,y)\bigr) \subset \CC[x,y]
\]
is the unit ideal. 
\par 
We will prove that $(F^{(n)},G^{(m)},J)=(1)$ by proving the following two
assertions.
\begin{itemize}
\setlength{\itemsep}{0pt}
\item
All solutions~$(\a,\b)$ to $F^{(n)}=G^{(m)}=0$ are
$3$-adically integral.
\item
$J(x,y)\equiv 1\pmod3$.
\end{itemize}

\begin{remark}
Our proof of Theorem~\ref{thm:cubicrigid}, \emph{mutatis mutandis},
can be used to show the following more general result.  Let~$p\ge3$ be
prime and let
\[
  f_{x,y}(z)=z^p-px^{p-1}z-y.
\]
The critical points of~$f_{x,y}$ are the points $\z x$,
where $\z\in\bfmu_{p-1}$. Let $\z_1$ and $\z_2$ be distinct
$(p-1)^{\text{st}}$-roots of unity. Fix integers $n,m\ge1$. Then the curves
\[
  f_{x,y}^n(\z_1 x)=\z_1 x\qquad\text{and}\qquad f_{x,y}^m(\z_2 x)=\z_2 x
\]
intersect transversally.
\end{remark}

We begin with a lemma that describes the iterates of~$f_{x,y}(z)$ evaluated
at $z=x$.

\begin{lemma}
\label{lemma:fxyiterates}
Let
\[
  f_{x,y}(z)=z^3-3x^2z+y.
\]
Then
\begin{equation}
  \label{eqn:fminusx}
  f^n_{x,y}(z) = f^n_{-x,y}(z),
\end{equation}
The iterates of~$f_{x,y}$ evaluated at~$x$ have the following
properties\textup:
\begin{parts}
\Part{(a)}
As a polynomial in~$x$,
\begin{equation}
  \label{eqn:fnxyxexpand}
  f^n_{x,y}(x)= \sum_{k=0}^{3^n} a_k(y)x^{3^n-k}
       \in \ZZ[y][x]
\end{equation}
with
\begin{align}
  \deg a_k(y) &\le 4\left\lfloor\frac{k}{3}\right\rfloor-k
    \label{eqn:aij1}\\
   a_0(y) &= (-2)^{3^{n-1}}\equiv 1 \pmod3, 
    \label{eqn:aij2}\\
   a_{3^n}(y) &= y^{3^{n-1}} + \textup{(lower order terms)}.
    \label{eqn:aij3}
\end{align}
\textup(By convention, a polynomial with negative degree is the zero
polynomial.\textup)
\Part{(b)}
Reducing modulo~$3$, we have
\[
  f^n_{x,y}(x) \equiv  x^{3^n} + y+y^3+y^9+\cdots+y^{3^{n-1}} \pmod{3}.
\]
\Part{(c)}
For $n,m\ge1$, define
\[
  F^{(n)}(x,y)=f^n_{x,y}(x)-x
  \qquad\text{and}\qquad
  G^{(m)}(x,y)=f^m_{x,y}(-x)+x.
\]
Then~$F^{(n)}(x,y)\in\ZZ[x,y]$ and~$G^{(m)}(x,y)\in\ZZ[x,y]$.
Further
\[
  G^{(m)}(x,y) = F^{(m)}(-x,y),
\]
and
\begin{align*}
  F^{(n)}(x,y) &\equiv x^{3^n} - x + y+y^3+y^9+\cdots+y^{3^{n-1}} \pmod{3},\\
  G^{(m)}(x,y) &\equiv -x^{3^m} + x + y+y^3+y^9+\cdots+y^{3^{m-1}} \pmod{3}.
\end{align*}
\end{parts}
\end{lemma}

\begin{remark}
The upper bound in the right-hand side of~\eqref{eqn:aij1} has the form
\[
  \begin{array}{|c*{5}{||r|r|r}|} \hline
    k & 0 & 1 & 2 & 3 & 4 & 5 & 6 & 7 & 8 & 9 & 10 & 11 & 12 & \dots\\ \hline
    4\lfloor k/3\rfloor-k &
      0 & -1 & -2 & 1 & 0 & -1 & 2 & 1 & 0 & 
      3 & 2 & 1 & 4 & \dots \\ \hline
  \end{array}
\]
Experimentally, it seems that the polynomials~$a_k(y)$ appearing in
the expansion~\eqref{eqn:fnxyxexpand} satisfy
\[
  \deg a_k(y) = 4\lfloor k/3\rfloor-k
  \quad\text{for all $k$ except $k=3^n-1$,}
\]
and $a_{3^n-1}(y)=0$. It would probably not be hard to prove
this by induction.
\end{remark}

\begin{proof}
It is clear that we can write~$f_{x,y}^n(x)$ in the
form~\eqref{eqn:fnxyxexpand} for some polynomials~$a_k(y)\in\ZZ[y]$,
so it remains to prove that these polynomials
satisfy~\eqref{eqn:aij1},~\eqref{eqn:aij2} and~\eqref{eqn:aij3}.  We
begin with the proof of~\eqref{eqn:aij1}, which is by induction
on~$n$.  To indicate the dependence on~$n$, we write $a_k^{(n)}(y)$.
For $n=1$ we have
\[
  f_{x,y}^1(x) = z^3-3x^2z+y\big|_{z\to x}=-2x^3+y,
\]
so
\[
  a_{0}^{(1)}(y)=-2,\quad a_{1}^{(1)}(y)=a_{2}^{(1)}(y)=0,\quad a_{3}^{(1)}(y)=y.
\]
\par
Next we assume that~\eqref{eqn:aij1} is true for~$n$ and we compute
\begin{align}
  \label{eqn:akinduct}
  f_{xy}^{n+1}(x)
  &= f_{xy}\bigl(f^n_{xy}(x)\bigr) \notag\\
  &= f^n_{xy}(x)^3 -3x^2f^n_{xy}(x) + y \notag\\
  &= \left(\sum_{k=0}^{3^n} a^{(n)}_k(y)x^{3^n-k}\right)^3
    -3x^2\left(\sum_{k=0}^{3^n} a^{(n)}_k(y)x^{3^n-k}\right) + y.
\end{align}
We consider first the cubed expression in~\eqref{eqn:akinduct}.
If it is multiplied out, we obtain a sum of terms of the form
\[
  a^{(n)}_i(y)x^{3^n-i}a^{(n)}_j(y)x^{3^n-j}a^{(n)}_k(y)x^{3^n-k} 
  = a^{(n)}_i(y)a^{(n)}_j(y)a^{(n)}_k(y)x^{3^{n+1}-i-j-k} 
\]
with $0\le i,j,k\le 3^n$.  Applying~\eqref{eqn:aij1} to
$a^{(n)}_i(y)$,~$a^{(n)}_j(y)$, and~$a^{(n)}_k(y)$, we find that
\begin{align*}
  \deg\bigl(a^{(n)}_i(y)a^{(n)}_j(y)a^{(n)}_k(y)\bigr)
  &\le 4\left\lfloor\frac{i}{3}\right\rfloor - i
   +  4\left\lfloor\frac{j}{3}\right\rfloor - j
   +  4\left\lfloor\frac{k}{3}\right\rfloor - k \\
  &\le 4\left\lfloor\frac{i+j+k}{3}\right\rfloor - i-j-k,
\end{align*}
where the last line follows from the elementary 
inequality (see Section~\ref{section:elemineq})
\begin{equation}
  \label{eqn:floorineq}
  \lfloor t_1 \rfloor + \lfloor t_2 \rfloor 
    + \lfloor t_3 \rfloor
  \le \lfloor t_1+t_2+t_3 \rfloor
  \quad\text{for all $t_1,t_2,t_3\in\RR$.}
\end{equation}
Thus terms coming from the cubed expression in~\eqref{eqn:akinduct}
satisfy~\eqref{eqn:aij1} for $n+1$. Since it
is easy to see that the other terms in~\eqref{eqn:akinduct}
satisfy~\eqref{eqn:aij1} for $n+1$, this completes the proof
by induction that~\eqref{eqn:aij1} holds for all~$n\ge1$.
\par
In order to prove~\eqref{eqn:aij2}, we observe that if we assign
weight~$1$ to both~$x$ and~$z$ and weight~$0$ to~$y$, then the terms
of weight~$3^n$ in~$f_{xy}^n(z)$ are precisely the ones that come from
repeatedly cubing the degree~$3$ expression~$z^3-3x^2z$, i.e.,
\[
  f^n_{xy}(z) = (z^3-3x^2z)^{3^{n-1}} + \textup{(lower weight terms)}.
\]
Hence
\begin{align*}
  a_{0}^{(n)}(y) 
  &= \textup{coefficient of $x^{3^n}$ in $f^n_{xy}(x)$}\\
  &= \textup{coefficient of $x^{3^n}$ in $(-2x^3)^{3^{n-1}}$}\\
  &= (-2)^{3^{n-1}}.
\end{align*}
\par
The proof of~\eqref{eqn:aij3} is a trivial induction on~$n$. More
precisely, if we let~$y$ have weight~$1$ and~$x$ and~$z$ have weight~$0$,
then
\begin{align*}
  f^{n+1}_{x,y}(x) 
  &= f^n_{x,y}(x)^3 - 3x^2 f^n_{x,y}(x) + y \\
  &= (y^{3^{n-1}}+\textup{(lower weight terms)})^3  \\
   &\qquad {} - 3x^2(y^{3^{n-1}} + \textup{(lower weight terms)}) + y \\
  &= y^{3^{n}}+\textup{(lower weight terms)}.
\end{align*}
This completes the proof of~(a).
\par
For~(b) we are working modulo~$3$, so
\[
  f_{x,y}(z) \equiv z^3 + y \pmod{3}.
\]
An easy induction gives the desired result,
\begin{align*}
  f^{n+1}_{x,y}(x) 
  &\equiv f\bigl(f^n_{x,y}(x)\bigr) \pmod3\\
  &\equiv f^n_{x,y}(x)^3 + y \pmod3\\
  &\equiv  (x^{3^n}+y+y^3+y^9+\cdots+y^{3^{n-1}})^3 + y \pmod3\\
  &\equiv x^{3^{n+1}}+y^3+y^9+y^{27}+\cdots+y^{3^{n}} + y \pmod3.
\end{align*}
\par
To prove the first part of~(c), we evaluate~\eqref{eqn:fminusx}
at~$z=-x$ to obtain $f^n_{x,y}(-x)=f^n_{-x,y}(-x)$. Substituting this
into the definition of~$G^{(n)}(x,y)$ yields
\[
  G^{(n)}(x,y) = f^n_{x,y}(-x)+x
  = f^n_{-x,y}(-x)+x
  = F^{(n)}(-x,y).
\]
Finally, the values of~$F^{(n)}$ and~$G^{(n)}$ modulo~$3$ follow from
the value of~$f_{x,y}^n(x)$ modulo~$3$ computed in~(b).
\end{proof}

An immediate consequence of Lemma~\ref{lemma:fxyiterates} is the
mod~$3$ value of the Jacobian.

\begin{proposition}
\label{prop:Jxy=1mod3}
The Jacobian determinant $J(x,y)\in\ZZ[x,y]$ defined
by~\eqref{eqn:Jxy} satisfies
\[
  J(x,y) \equiv 1 \pmod{3}.
\]
\end{proposition}
\begin{proof}
Differentiating the formulas for~$F^{(n)}(x,y)$ and~$G^{(m)}(x,y)$ in
Lemma~\ref{lemma:fxyiterates}(c) and reducing modulo~$3$ yields
\[
  J(x,y)
  =   \det\begin{pmatrix} 
    F^{(n)}_x(x,y) & G^{(m)}_x(x,y) \\
    F^{(n)}_y(x,y) & G^{(m)}_y(x,y) \\
   \end{pmatrix}
  \equiv \det\begin{pmatrix} 
      -1 & 1 \\ 1 & 1 \\
   \end{pmatrix}
  \equiv 1 \pmod 3.
\]
\end{proof}

Before tackling the $3$-integrality of the common roots
of~$F^{(n)}(x,y)$ and $G^{(m)}(x,y)$, we prove two elementary lemmas.
With an eye towards generalizations, we work over~$\FF_p$.

\begin{lemma}
\label{lemma:prodTuv}
Let $p$ be a prime, let $m,n\ge1$ be integers, let $d=\gcd(m,n)$, and
let~$\tau$ denote the $p$-power Frobenius map. Then
\begin{equation}
  \label{eqn:prodTuvtdtntmtd}
  \prod_{u\in\FF_{p^n}} \prod_{v\in\FF_{p^m}} (T-u-v)
  = \frac{\tau^d\circ(\tau^n-1)\circ(\tau^m-1)}{\tau^d-1}(T)\in\FF_p[T].
\end{equation}
\end{lemma}

\begin{remark}
The meaning of the right-hand side of~\eqref{eqn:prodTuvtdtntmtd} is
as follows. The rational expression
$\frac{\tau^d\circ(\tau^n-1)\circ(\tau^m-1)}{\tau^d-1}$ is actually a
polynomial in~$\tau$, since~$d$ divides~$m$. In other words, it is an
element of~$\ZZ[\tau]$. We then use the natural action of~$\ZZ[\tau]$
on~$\FF_p[T]$ defined by
\[
  \left(\sum a_i\tau^i\right)\bigl(f(T)\bigr)
  = \sum a_if(T)^{p^i}.
\]
\end{remark}

\begin{proof}
We first observe that if $u_1,u_2\in\FF_{p^n}$ and $v_1,v_2\in\FF_{p^m}$
satisfy
\[
  u_1+v_1 = u_2+v_2,
\]
then
\[
  u_1-u_2 = v_1-v_2 \in \FF_{p^n}\cap\FF_{p^m} = \FF_{p^d}.
\]
Hence
\[
  \prod_{u\in\FF_{p^n}} \prod_{v\in\FF_{p^m}} (T-u-v)
  = \biggl(\prod_{w\in(\FF_{p^n}+\FF_{p^m})/\FF_{p^d}} (T-w)\biggr)^{p^d}.
\]
Let
\[
  \f(T)=\prod_{w\in(\FF_{p^n}+\FF_{p^m})/\FF_{p^d}} (T-w)
  \quad\text{and}\quad
  \psi(T)=\frac{(\tau^n-1)\circ(\tau^m-1)}{\tau^d-1}(T).
\]
Our earlier observation shows that~$\f(T)$ has distinct roots, and it
is monic of degree~$p^{n+m-d}$. 
\par
We next observe that for any $u\in\FF_{p^n}$ and $v\in\FF_{p^m}$,
we have
\begin{align*}
  \psi(u+v) 
  &= \frac{(\tau^n-1)\circ(\tau^m-1)}{\tau^d-1}(u+v) \\
  &= \left(\frac{\tau^m-1}{\tau^d-1}\right)\circ(\tau^n-1)(u)
      + \left(\frac{\tau^n-1}{\tau^d-1}\right)\circ(\tau^m-1)(v) \\
  &= 0,
\end{align*}
since $\tau^n(u)=u$ and $\tau^m(v)=v$. Thus~$\psi(T)$ vanishes at each of
the roots of~$\f(T)$, and~$\f(T)$ has simple roots, 
so~\text{$\f(T)\mid\psi(T)$}. But~$\psi(T)$ is monic
and has the same degree~$p^{n+m-d}$ as~$\f(T)$. Hence~$\psi(T)=\f(T)$,
and therefore
\begin{align*}
  \prod_{u\in\FF_{p^n}} \prod_{v\in\FF_{p^m}} (T-u-v)
  &= \f(T)^{p^d} = \psi(T)^{p^d} = \tau^d\bigl(\psi(T)\bigr) \\
  &= \frac{\tau^d\circ(\tau^n-1)\circ(\tau^m-1)}{\tau^d-1}(T).
\end{align*}
This completes the proof of Lemma~\ref{lemma:prodTuv}.
\end{proof}

\begin{lemma}
\label{lemma:resxpnxpm}
Let $p$ be a prime, let $m,n\ge1$ be integers, let $d=\gcd(m,n)$, and
let~$\tau$ denote the $p$-power Frobenius map. Then working
in~$\FF_p[A,B]$, we have
\[
  \Resultant(x^{p^n}-x-A,x^{p^m}-x-B)
  = \frac{\tau^d\circ(\tau^m-1)}{\tau^d-1}(A)
    -  \frac{\tau^d\circ(\tau^n-1)}{\tau^d-1}(B).
\]
\end{lemma}

\begin{remark}
Lemma~\ref{lemma:resxpnxpm} uses Frobenius to give a compact
expression for the resultant, but we can also write it out explicitly as
\[
  \Resultant(x^{p^n}-x-A,x^{p^m}-x-B)
  = \sum_{i=1}^{m/d} A^{p^{id}} - \sum_{i=1}^{n/d} B^{p^{id}}.
\]
\end{remark}

\begin{proof}
Let~$\a,\b\in\overline{\FF_p(A,B)}$ be roots, respectively, of
\[
  x^{p^n}-x-A\quad\text{and}\quad x^{p^m}-x-B.
\]
The extensions $\bar\FF_p(\a)/\bar\FF_p(A)$ and
$\bar\FF_p(\b)/\bar\FF_p(B)$ are Artin--Scheier extensions. 
The conjugates of~$\a$ over~$\bar\FF_p(A)$ are
\[
  \bigl\{\a+u : u\in\FF_{p^n}\bigr\},
\]
and similarly for~$\b$, so we have factorizations
\[
  x^{p^n}-x-A=\prod_{u\in\FF_{p^n}} (x-\a-u)
  \quad\text{and}\quad 
  x^{p^m}-x-B=\prod_{v\in\FF_{p^m}} (x-\b-v).
\]
We now compute
\begin{align*}
  \Resultant&(x^{p^n}-x-A,x^{p^m}-x-B) \\
  &=  \prod_{u\in\FF_{p^n}} \prod_{v\in\FF_{p^m}} (\a+u-\b-v) 
    &&\text{see \cite[2.13(b)]{MR2316407},} \\
  &=  \prod_{u\in\FF_{p^n}} \prod_{v\in\FF_{p^m}} \bigl((\a-\b)-u-v) \\
  &= \frac{\tau^d\circ(\tau^n-1)\circ(\tau^m-1)}{\tau^d-1}(\a-\b) 
    &&\text{from Lemma \ref{lemma:prodTuv},} \\
  &= \frac{\tau^d\circ(\tau^m-1)}{\tau^d-1}(A) 
   -  \frac{\tau^d\circ(\tau^n-1)}{\tau^d-1}(B) 
    &&\begin{tabular}[t]{l}
      since $(\tau^n-1)(\a)=A$\\ and $(\tau^m-1)(\b)=B$.\\
      \end{tabular}
\end{align*}
This completes the proof of Lemma~\ref{lemma:resxpnxpm}.
\end{proof}

\begin{remark}
We observe that for $m=n$, Lemma~\ref{lemma:resxpnxpm} can be proven
directly from the Sylvester matrix. To ease notation, let $N=p^n$.
Then the Sylvester matrix for the resultant of $x^N-x-A$ and $x^N-x-B$
is the $2N$-by-$2N$ matrix
\begin{equation}
  S(A,B) = 
  \left[
  \begin{array}{*{10}c}
    1 & 0 & 0 & -1 & A \\
      & 1 & 0 & 0 & -1 & A \\
      &   & \ddots & & & & \ddots \\
      & & & 1 & 0 & 0 & -1 & A \\
    1 & 0 & 0 & -1 & B \\
      & 1 & 0 & 0 & -1 & B \\
      &   & \ddots & & & & \ddots \\
      & & & 1 & 0 & 0 & -1 & B \\
  \end{array}
  \right].
\end{equation}
If we subtract each row in the top half from the corresponding row in
the bottom half, we obtain an upper-triangular matrix whose diagonal
is $(1,1,\ldots,1,B-A,B-A,\ldots,B-A)$. Hence
\[
  \Resultant(x^{p^n}-x-A,x^{p^m}-x-B)
  = \det S(A,B) = (B-A)^N = B^{p^n} - A^{p^n}.
\]
\end{remark}

\begin{proposition}
\label{prop:ResFnGm3int}
Let $F^{(n)}(x,y)$ and $G^{(m)}(x,y)$ be as defined
by~\eqref{eqn:defFG}. Then
\[
  \Resultant_x\bigl(F^{(n)}(x,y),G^{(m)}(x,y)\bigr) \in \ZZ[y]
\]
is a polynomial of degree $3^{n+m-1}$ with integer coefficients and
leading coefficient relatively prime to~$3$.
\end{proposition}
\begin{proof}
As in~\eqref{eqn:fnxyxexpand} of Lemma~\ref{lemma:fxyiterates}(a), we
write
\[
  f^n_{x,y}(x)= \sum_{k=0}^{3^n} a_k(y)x^{3^n-k}
       \in \ZZ[y][x]
\]
with  polynomials~$a_k(y)$
satisfying~\eqref{eqn:aij1},~\eqref{eqn:aij2},
and~\eqref{eqn:aij3}. We similarly write
\[
  f^m_{x,y}(-x)= \sum_{k=0}^{3^n} b_k(y)x^{3^m-k}
       \in \ZZ[y][x].
\]
(We adopt this notation as being less clumsy for the present proof
than our earlier notation, which would have been $a_k(y)=a_k^{(n)}(y)$
and $b_k(y)=(-1)^{i+1}a_k^{(m)}(y)$.) Then
\[
  F^{(n)}(x,y) = \sum_{k=0}^{3^n} a_k(y)x^{3^n-k}-x
  \quad\text{and}\quad
  G^{(m)}(x,y) = \sum_{k=0}^{3^m} b_k(y)x^{3^m-k}+x.
\]
In order to surpress the extra~$\pm x$ for the moment, we write
\[
  F^{(n)}(x,y) = \sum_{k=0}^{3^n} A_k(y)x^{3^n-k}
  \quad\text{and}\quad
  G^{(m)}(x,y) = \sum_{k=0}^{3^m} B_k(y)x^{3^m-k},
\]
where $A_k=a_k$ except $A_{3^n-1}=a_{3^n-1}-1$, and similarly for~$B_k$.
We observe that the degree estimates for~$a_k$ given
by~\eqref{eqn:aij1} are true for~$A_k$ and~$B_k$, since the extra~$\pm x$
is within the specified bound for the degree.
\par
To ease notation, we let
\[
  N=3^n\qquad\text{and}\qquad M=3^m.
\]
Then the $x$-resultant of~$F^{(n)}(x,y)$ and~$G^{(m)}(x,y)$ is given
by the determinant of the Sylvester matrix
\begin{equation*}
  \left[
  \begin{array}{*{10}c}
    A_0 & A_1 & A_2 & A_3 & \cdots & A_{N-1} & A_N \\
      & A_0 & A_1 & A_2 & A_3 & \cdots & A_{N-1} & A_N \\
      &   & \ddots & & & & & & \ddots \\
      & & & A_0 & A_1 & A_2 & A_3 & \cdots & A_{N-1} & A_N \\
    B_0 & B_1 & \cdots & B_{M-1} & B_M \\
      & B_0 & B_1 & \cdots & B_{M-1} & B_M \\
      &   & \ddots & & & & \ddots \\
      & & & B_0 & B_1 & \cdots & B_{M-1} & B_M \\
      & & & & B_0 & B_1 & \cdots & B_{M-1} & B_M \\
      & & & & & B_0 & B_1 & \cdots & B_{M-1} & B_M \\
  \end{array}
  \right].
\end{equation*}
The Sylvester matrix, which we denote by~$S$, is a square matrix of
size~$M+N$. Its top~$M$ rows have~$A_k$ coefficients and its
bottom~$N$ rows have~$B_k$ coefficients. When we entirely expand $\det
S$, it is a sum of terms of the form
\[
  (-1)^{\operatorname{sign}(\s)}\prod_{i=1}^{M+N} S_{i,\s(i)},
\]
where~$\s$ is a permutation of~$\{1,2,\ldots,M+N\}$.  We are
interested in bounding the degree of this term, so we assume that all
of the $S_{i,\s(i)}$ are nonzero and compute
\begin{align*}
  \deg\biggl(\prod_{i=1}^{M+N} S_{i,\s(i)}\biggr)
  &= \sum_{i=1}^M \deg(S_{i,\s(i)}) + \sum_{i=M+1}^{M+N} \deg(S_{i,\s(i)}) \\
  &= \sum_{i=1}^M \deg(A_{\s(i)-i}) + \sum_{i=M+1}^{M+N} \deg(B_{\s(i)-(i-M)})\\
  &= \sum_{i=1}^M \deg(A_{\s(i)-i}) + \sum_{i=1}^{N} \deg(B_{\s(i+M)-i}).
\end{align*}
We now apply the bound~\eqref{eqn:aij1} from
Lemma~\ref{lemma:fxyiterates}(a), which as we noted earlier applies
to~$A_k$ and~$B_k$. This yields
\begin{align*}
  \deg\biggl(\prod_{i=1}^{M+N} S_{i,\s(i)}\biggr)
  &\le \sum_{i=1}^M \left(4\left\lfloor\frac{\s(i)-i}{3}\right\rfloor
             -(\s(i)-i)\right) \\
  &\qquad{}
    +  \sum_{i=1}^N \left(4\left\lfloor\frac{\s(i+M)-i}{3}\right\rfloor
             -(\s(i+M)-i)\right).
\end{align*}
We rewrite this last expression using fractional part notation,
\[
  \{t\}=t-\lfloor t\rfloor,
\]
to obtain
\begin{align*}
  \deg\biggl(\prod_{i=1}^{M+N} S_{i,\s(i)}\biggr)
  &\le \sum_{i=1}^M \left(\frac{\s(i)-i}{3}
        -4\left\{\frac{\s(i)-i}{3}\right\}\right) \\
  &\qquad{}
    +  \sum_{i=1}^N \left(\frac{\s(i+M)-i}{3}
        -4\left\{\frac{\s(i+M)-i}{3}\right\}\right) \\
  &= \frac13\biggl(\sum_{j=1}^{M+N} j - \sum_{i=1}^M i - \sum_{i=1}^N i\biggr)\\
  &\qquad{}
    -  4\sum_{i=1}^M \left\{\frac{\s(i)-i}{3}\right\}
    -  4\sum_{i=1}^N \left\{\frac{\s(i+M)-i}{3}\right\} \\
  &=\frac{MN}{3} 
    -  4\sum_{i=1}^M \left\{\frac{\s(i)-i}{3}\right\}
    -  4\sum_{i=1}^N \left\{\frac{\s(i+M)-i}{3}\right\} \\
  &\le \frac{MN}{3} = 3^{m+n-1}.
\end{align*}
Since the determinant of the Sylvester matrix is a sum of terms of this
form, we have proven that
\[
  \deg(\det S) \le 3^{m+n-1}.
\]
\par
We are next going to evaluate $\det S$ modulo~$3$.
To ease notation, we let
\[
  Y_n = y+y^3+y^9+\cdots+y^{3^{n-1}}
  \quad\text{and}\quad
  Y_m = y+y^3+y^9+\cdots+y^{3^{m-1}}.
\]
Then Lemma~\ref{lemma:fxyiterates}(c) says that
\begin{align*}
  F^{(n)}(x,y)&\equiv x^N-x+Y_n \pmod3,\\
  G^{(m)}(x,y)&\equiv -x^M+x+Y_m \pmod3.
\end{align*}
Working modulo~$3$, this allows us to compute
\begin{align*}
  \Resultant_x\bigl(F^{(n)}(x,y)&,G^{(m)}(x,y)\bigr)\\
  &\equiv  \Resultant_x(x^N-x+Y_n,-x^M+x+Y_m) \pmod3\\
  &\equiv - \Resultant_x(x^N-x+Y_n,x^M-x-Y_m) \pmod3.
\end{align*}
We apply Lemma~\ref{lemma:resxpnxpm} with $A=-Y_n$ and
$B=Y_m$. Letting $d=\gcd(m,n)$ and $\tau$ denote $3$-power Frobenius,
this gives
\begin{align*}
  \det(S)
  &=\Resultant_x\bigl(F^{(n)}(x,y),G^{(m)}(x,y)\bigr) \\
  &\equiv -\frac{\tau^d\circ(\tau^m-1)}{\tau^d-1}(-Y_n)
    +  \frac{\tau^d\circ(\tau^n-1)}{\tau^d-1}(Y_m)  \pmod3\\
  &\equiv \t^m(Y_n)+\t^{m-d}(Y_n)+\cdots+\t^d(Y_n) \\
  &\qquad{} + \t^n(Y_m)+\t^{n-d}(Y_m)+\cdots+\t^d(Y_m) \pmod3\\
  &\equiv Y_n^{3^m} + Y_m^{3^n}  + \textup{(lower order terms)} \pmod3 \\
  &\equiv 2y^{3^{m+n-1}}  + \textup{(lower order terms)} \pmod3.
\end{align*}
\par
We have now proven that
\[
  \deg(\det S) \le 3^{m+n-1} 
  \quad\text{and}\quad
  \det S \equiv 2y^{3^{m+n-1}}  + \textup{(l.o.t.)} \pmod3.
\]
It follows that~$\det S$ has degree exactly equal to~$3^{m+n-1}$
and that its leading coefficient is relatively prime to~$3$,
which completes the proof of Proposition~\ref{prop:ResFnGm3int}.
\end{proof}

We now have all of the tools needed to prove
Theorem~$\ref{thm:cubicrigid}$.

\begin{proof}[Proof of Theorem~$\ref{thm:cubicrigid}$]
Let $(\a,\b)$ be a solution to
\[
  F^{(n)}(x,y)=G^{(m)}(x,y)=0.
\]
To ease notation, let
\[
  R^{(n,m)}(y)=\Resultant_x\bigl(F^{(n)}(x,y),G^{(m)}(x,y)\bigr).
\]
A standard property of the resultant of two polynomials says that it
is in the ideal generated by those
polynomials~\cite[2.13(c)]{MR2316407}. Thus there are
polynomials $U(x,y),V(x,y)\in\ZZ[x,y]$ such that
\[
  U(x,y)F^{(n)}(x,y)+V(x,y)G^{(m)}(x,y) 
    = R^{(n,m)}(y).
\]
Substituting~$(x,y)=(\a,\b)$, we find that $R^{(n,m)}(\b)=0$.
Proposition~\ref{prop:ResFnGm3int} says that $R^{(n,m)}(y)\in\ZZ[y]$
has leading coefficient prime to~$3$, which proves that~$\b$ is
$3$-adically integral. We next use Lemma~\ref{lemma:fxyiterates}(a) to
write
\[
  F^{(n)}(x,y)
  = f_{x,y}^n(x)-x = (-2)^{3^{n-1}}x^{3^n} + \sum_{k=1}^{3^n} a_k(y)x^{3^n-k} - x.
\]
Substituting $y=\b$ we see that~$\a$ is a root of the
polynomial $F^{(n)}(x,\b)$ whose coefficients are $3$-adically integral
and whose leading coefficient is a $3$-adic unit. Hence~$\a$ is also
$3$-adically integral.
\par
Now consider the value~$J(\a,\b)$ of the Jacobian
determinant~\eqref{eqn:Jxy}. Proposition~\ref{prop:Jxy=1mod3}
says that there is a polynomial $K(x,y)\in\ZZ[x,y]$ satisfying
\[
  J(x,y) = 1 + 3K(x,y).
\]
We know that~$\a$ and~$\b$ are $3$-adically integral, so the same
is true of~$J(\a,\b)$ and~$K(\a,\b)$. Taking norms down to~$\QQ$,
we find that
\[
  \Norm_{\QQ(\a,\b)/\QQ}  J(\a,\b)=\Norm_{\QQ(\a,\b)/\QQ} \bigl(1+3K(\a,\b)\bigr)
  \equiv 1 \pmod3.
\]
In particular,~$J(\a,\b)\ne0$.
It follows that the ideal
\[
  \bigl(F^{(n)}(x,y),G^{(m)}(x,y),J(x,y)\bigr) \subset \CC[x,y]
\]
is the unit ideal, since if it weren't, then~$F^{(n)}$,~$G^{(n)}$,
and~$J$ would have a common root.  This completes the proof that the
curves $F^{(n)}(x,y)=0$ and $G^{(m)}(x,y)=0$ intersect transversally.
\end{proof}

\section{Preperiodic critical points --- a modest beginning}
\label{section:preperiodic}

Generalizing the notation from Theorem~\ref{thm:cubicrigid}, 
for $i,j\ge1$ we let
\begin{align*}
  C_{1,n,i} &= \left\{(f,c_1,c_2)\in\Pcal_3^\crit :
     \begin{tabular}{l}
     $f^{i+n}(c_1)=f^i(c_1)$ and \\ $f^{i-1+n}(c_1)\ne f^{i-1}(c_1)$ \\
     \end{tabular}
  \right\}, \\
  C_{2,m,j} &= \left\{(f,c_1,c_2)\in\Pcal_3^\crit :
     \begin{tabular}{l}
     $f^{j+m}(c_2)=f^j(c_2)$ and\\ $f^{j-1+m}(c_2)\ne f^{j-1}(c_2)$\\
     \end{tabular}
  \right\} .
\end{align*}
In words, $(f,c_1,c_2)\in C_{1,n,i}$ if~$c_1$ is purely preperiodic
with tail length~$i$ and cycle length dividing~$n$, and similarly for
$C_{2,m,j}$. For convenience, we let $C_{1,n,0}=C_{1,n}$ and
$C_{2,m,0}=C_{2,m}$.

Thurston's theorem implies that $C_{1,n,i}$ and $C_{2,m,j}$ intersect
transversally. We sketch a $3$-adic proof of a very special case.
The key to the proof is the following elementary identity.

\begin{lemma}
\label{lemma:fxyn1x}
We have
\begin{equation}
  \label{eqn:fn1xfxfactors}
  f_{x,y}^{n+1}(x) - f_{x,y}(x)
  = \bigl(f_{x,y}^n(x)-x\bigr)^2\bigl(f_{x,y}^n(x)+2x\bigr).
\end{equation}
In particular, we have
\[
  F^{(n,1)}(x,y) 
  = \frac{f_{x,y}^{n+1}(x) - f_{x,y}(x)}{\bigl(f_{x,y}^n(x)-x\bigr)^2}
  \in \ZZ[x,y],
\]
and $F^{(n,1)}(x,y)$ satisfies
\begin{align*}
  F^{(n,1)}(x,y) &\equiv F^{(n)}(x,y) \pmod3,\\
  F_x^{(n,1)}(x,y) &\equiv 1 \pmod3,\\
  F_y^{(n,1)}(x,y) &\equiv -1 \pmod3.
\end{align*}
\end{lemma}
\begin{proof}
The polynomial~$f_{x,y}(z)$ has a critical point at~$z=x$, so
the difference $f_{x,y}(z)-f_{x,y}(x)$ should be divisible by~$(z-x)^2$.
Explicitly, we find that
\[
  f_{x,y}(z)-f_{x,y}(x) = (z-x)^2(z+2x).
\]
Substituting $z=f_{x,y}^n(x)$ gives~\eqref{eqn:fn1xfxfactors}.
Then the function we have called $F^{(n,1)}(x,y)$ is given by
\[
  F^{(n,1)}(x,y) = f_{x,y}^n(x)+2x.
\]
Reducing modulo~$3$ gives
\[
  F^{(n,1)}(x,y) = f_{x,y}^n(x)+2x \equiv f_{x,y}^n(x)-x = F^{(n)}(x,y) \pmod3.
\]
The formulas for the partial derivatives of $F^{(n,1)}(x,y)\bmod3$
then follow by differentiating the formula for~$F^{(n)}(x,y)$ given in
Lemma~\ref{lemma:fxyiterates}(c).
\end{proof}

\begin{theorem}
\begin{parts}
\Part{(a)}
$C_{1,n,1}$ and $C_{2,m,0}$ intersect transversally.
\Part{(a)}
$C_{1,n,1}$ and $C_{2,m,1}$ intersect transversally.
\end{parts}
\end{theorem}
\begin{proof}
As usual, let
\[
  f_{x,y}(z)= z^3 - 3x^2z + y
\]
be a cubic polynomial normalized to have critical points $\pm x$.
Then the points $(f_{x,y},x,-x)$ in~$C_{1,n,1}$ are the points
satisfying
\[
  f_{x,y}^{n+1}(x)=f_{x,y}(x)\quad\text{and}\quad f_{x,y}^n(x)\ne x.
\]
From Lemma~\ref{lemma:fxyn1x}, these points satisfy
\[
  F^{(n,1)}(x,y) = f_{x,y}^n(x)+2x = 0,
\]
so the locus $F^{(n,1)}(x,y)=0$ contains the curve~$C_{1,n,1}$.
We will show that the curves
\[
  F^{(n,1)}(x,y)=0\qquad\text{and}\qquad G^{(m)}(x,y)=0
\]
intersect transversally.
\par
The first part of the proof is to show that the intersection points
are $3$-adically integral. This can be proven using the resultant
methods, \emph{mutatis mutandis}, of this paper. It is also proven in
a more general setting by Epstein~\cite{arxiv1010.2780}.  We then
compute the Jacobian using the congruences for the
derivatives~$F_x^{(n,1)}(x,y)$ and~$F_y^{(n,1)}(x,y)$ given in
Lemma~\ref{lemma:fxyn1x} and differentiating the formula for
$G^{(m)}(x,y)$ given in Lemma~\ref{lemma:fxyiterates}(c). Thus
\begin{align*}
  J(x,y)
  &=   \det\begin{pmatrix} 
    F^{(n,1)}_x(x,y) & G^{(m)}_x(x,y) \\
    F^{(n,1)}_y(x,y) & G^{(m)}_y(x,y) \\
   \end{pmatrix} \pmod3 \\
  &\equiv \det\begin{pmatrix} 
      -1 & 1 \\ 1 & 1 \\
   \end{pmatrix}
  \equiv 1 \pmod 3.
\end{align*}
\par
The proof for $C_{1,n,1}$ and $C_{2,m,1}$ is almost identical, since
replacing~$x$ by~$-x$ in Lemma~\ref{lemma:fxyn1x} gives
\[
  G^{(m,1)}_x(x,y) \equiv 1\pmod3
  \qquad\text{and}\qquad
  G^{(m,1)}_y(x,y) \equiv 1\pmod3.
\]
\end{proof}

\section{Proof of \eqref{eqn:floorineq}}
\label{section:elemineq}
For the convenience of the reader, we prove the elementary
inequality~\eqref{eqn:floorineq} used in the proof of
Lemma~\ref{lemma:fxyiterates}.
For $t\in\RR$, write $t=\lfloor t\rfloor+\{t\}$, where $0\le\{t\}<1$
is the fractional part of~$t$. Then~\eqref{eqn:floorineq} is equivalent to
the inequality
\[
  \{t_1+t_2+t_3\} \le \{t_1\}+\{t_2\}+\{t_3\}.
\]
This inequality is invariant under $t_i\to t_i+k$ for any $k\in\ZZ$,
so without loss of generality, we may assume that $0\le t_i<1$ for
all $1\le i\le 3$. Then the desired inequality is
\[
  \{t_1+t_2+t_3\} \le t_1+t_2+t_3,
\]
which is trivially true. (There is nothing special about a sum of
three terms. The same proof shows that $\sum\lfloor t_i\rfloor\le
\lfloor\sum t_i\rfloor$.)

\begin{acknowledgement}
I would like to thank Adam Epstein for suggesting generalizing
Gleason's $2$-adic proof to prove other cases of Thurston's theorem,
and Adam Epstein and Bjorn Poonen for ongoing discussions of related
matters. I would also like to thank Xander Faber for organizing and
the CRM for funding the May, 2010 workshop on ``Moduli Spaces and the
Arithmetic of Dynamical Systems'' at the Bellairs Research Institute in
Barbados, where these discussions began.
\end{acknowledgement}





\begin{thebibliography}{1}

\bibitem{DouHubNotes}
A.~Douady and J.~H. Hubbard.
\newblock Exploring the {M}andelbrot set. {T}he {O}rsay notes.
\newblock \url{www.math.cornell.edu/~hubbard/OrsayEnglish.pdf}.

\bibitem{MR1251582}
A.~Douady and J.~H. Hubbard.
\newblock A proof of {T}hurston's topological characterization of rational
  functions.
\newblock {\em Acta Math.}, 171(2):263--297, 1993.

\bibitem{arxiv1010.2780}
A.~Epstein.
\newblock Integrality and rigidity for postcritically finite polynomials, 2010.
\newblock \url{arXiv:1010.2780}.

\bibitem{MR1620850}
C.~T. McMullen and D.~P. Sullivan.
\newblock Quasiconformal homeomorphisms and dynamics. {III}. {T}he
  {T}eichm\"uller space of a holomorphic dynamical system.
\newblock {\em Adv. Math.}, 135(2):351--395, 1998.

\bibitem{MR2316407}
J.~H. Silverman.
\newblock {\em The {A}rithmetic of {D}ynamical {S}ystems}, volume 241 of {\em
  Graduate Texts in Mathematics}.
\newblock Springer, New York, 2007.

\end{thebibliography}

\def\cprime{$'$}

\end{document}